\documentclass[11pt,reqno,tbtags]{amsart}
\pdfoutput=1
\usepackage[]{amsmath,amssymb,amsfonts,latexsym,amsthm,enumerate}
\usepackage{amssymb,bbm,paralist}
\usepackage{enumerate,graphicx}
\usepackage[font=small]{caption,subfig}
\usepackage[numeric,initials,nobysame]{amsrefs}

\numberwithin{equation}{section}
\usepackage{stmaryrd,bbm}

\newtheorem{maintheorem}{Theorem}

\newtheorem{theorem}{Theorem}[section]
\newtheorem*{theorem*}{Theorem}

\newtheorem{claim}[theorem]{Claim}
\newtheorem{proposition}[theorem]{Proposition}

\newtheorem{fact}[theorem]{Fact}
\newtheorem{corollary}[theorem]{Corollary}

\theoremstyle{definition}{

\newtheorem*{definition*}{Definition}

}

\theoremstyle{remark}{

\newtheorem*{remark*}{Remark}

}

\newcommand{\R}{\mathbb R}

\newcommand{\Z}{\mathbb Z}
\newcommand{\deq}{\stackrel{\scriptscriptstyle\triangle}{=}}

\newcommand{\E}{\mathbb{E}}
\renewcommand{\P}{\mathbb{P}}
\DeclareMathOperator{\var}{Var}

\newcommand{\gap}{\text{\tt{gap}}}

\newcommand{\tmix}{t_\textsc{mix}}

\newcommand{\tv}{{\textsc{tv}}}

\newcommand{\given}{\, \big| \,}
\newcommand{\one}{\mathbbm{1}}

\newcommand{\btop}{\textsc{t}}
\newcommand{\bmid}{\textsc{m}}
\newcommand{\llb}{\llbracket}
\newcommand{\rrb}{\rrbracket}
\newcommand{\OmFK}{\Omega_\textsc{fk}}
\newcommand{\OmI}{\Omega_\textsc{is}}
\newcommand{\Lbar}{\bar{\Lambda}}

\renewcommand{\epsilon}{\varepsilon}
\renewcommand{\phi}{\varphi}

\newcommand{\cM}{\mathcal{M}}
\newcommand{\cP}{\mathcal{P}}
\newcommand{\cB}{\mathcal{B}}

\newcommand{\cF}{\mathcal{F}}

\date{}

\begin{document}
\title[Mixing of critical Ising on the square lattice]{Critical Ising on the square lattice \\ mixes in polynomial time}

\author{Eyal Lubetzky}
\address{Eyal Lubetzky\hfill\break
Microsoft Research\\
One Microsoft Way\\
Redmond, WA 98052-6399, USA.}
\email{eyal@microsoft.com}
\urladdr{}

\author{Allan Sly}
\address{Allan Sly\hfill\break
Microsoft Research\\
One Microsoft Way\\
Redmond, WA 98052-6399, USA.}
\email{allansly@microsoft.com}
\urladdr{}

\begin{abstract}
The Ising model is widely regarded as the most studied model of spin-systems in statistical physics. The focus of this paper is its dynamic (stochastic) version, the Glauber dynamics, introduced in 1963 and by now the most popular means of sampling the Ising measure.
Intensive study throughout the last three decades has yielded a rigorous understanding of the spectral-gap of the dynamics on $\Z^2$ everywhere except at criticality. While the critical behavior of the Ising model has long been the focus for physicists, mathematicians have only recently developed an understanding of its critical geometry with the advent of SLE, CLE and new tools to study conformally invariant systems.

A rich interplay exists between the static and dynamic models.
At the static phase-transition for Ising, the dynamics is conjectured
to undergo a \emph{critical slowdown}: At high temperature the inverse-gap is $O(1)$, at the critical $\beta_c$ it is polynomial in the side-length and at low temperature it is exponential in it.
A seminal series of papers verified this on $\Z^2$ except at $\beta=\beta_c$
where the behavior remained a challenging open problem.

Here we establish the first rigorous polynomial upper bound for the critical mixing,
thus confirming the critical slowdown for the Ising model in $\Z^2$. Namely, we show that on a finite box with arbitrary (e.g.\ fixed, free, periodic) boundary conditions, the inverse-gap
at $\beta=\beta_c$ is polynomial in the side-length. The proof harnesses recent understanding of the scaling limit of critical Fortuin-Kasteleyn representation of the Ising model together with classical tools from the analysis of Markov chains.
\end{abstract}

\maketitle

\vspace{-0.77cm}

\section{Introduction}
The classical Ising model on the lattice is one of the most studied models in mathematical physics with thousands of research papers since its introduction in 1925. In his famous work from 1944, Onsager~\cite{Onsager} exactly solved the model in two dimensions, thereby determining its critical temperature.
Ever since, physicists directed most of their attention to the fascinating behavior of the model at criticality (see for instance the 20 volumes of~\cite{DL}). In this regime the model on $\Z^2$ exhibits delicate fractal geometry
whose rigorous understanding was only recently obtained using conformally invariant scaling limits.
Here we report the first rigorous confirmation that the spectral-gap of the Glauber dynamics for the critical Ising model in $\Z^2$, perhaps the most practiced methods for sampling its Gibbs distribution, is polynomial in the side-length.

The Glauber dynamics (also known as the \emph{stochastic Ising model}) is a family of Markov chains introduced by Glauber~\cite{Glauber} in 1963, which both models the dynamic evolution of the Ising model
and provides a simple algorithm for sampling from its stationary distribution. The most notable examples are the
heat-bath dynamics and metropolis dynamics, both highly used in practice thanks to their simple and natural transition-rules.

An extensive program of work by mathematicians, physicists and computer scientists has related the static spatial-mixing properties of the Ising model to the mixing-rate of the Glauber dynamics, measured in terms of the gap in the spectrum of its generator. Relying on many experiments and studies in the theory of dynamical critical phenomena,
the spectral-gap of the dynamics on a finite box in the lattice is conjectured to have the following \emph{critical slowdown} behavior (e.g.,~\cites{HH,LF,Martinelli97,WH}):
 \begin{compactitem}
 \item At high temperatures the inverse-spectral-gap is $O(1)$
  \item At the critical $\beta_c$ it is polynomial in the side-length
  \item At low temperatures it is exponential in the side-length
   \end{compactitem}
(in dimensions $d\geq3$ the surface-area plays the role of the side-length).
As we detail later, in a long series of seminal papers over the last three decades this rich interplay between the static and dynamic models was confirmed for the 2-dimensional Ising model with the crucial exception of the critical temperature $\beta_c$. By contrast to the detailed picture by now known for $\beta\neq \beta_c$, the behavior of the inverse-gap at criticality remained a stubborn and fundamental open problem, with no known sub-exponential bounds.

Our main result establishes that the inverse-gap is indeed polynomial in the side-length, thereby confirming the critical slowdown behavior of the Ising model on $\Z^2$.

\begin{maintheorem}\label{thm-1}
Consider the critical Ising model on a finite box $\Lambda \subset \Z^2$ of side-length $n$, i.e.\ at inverse-temperature $\beta_c=\frac12\log(1+\sqrt{2})$.
Let $\gap_\Lambda^\tau$ be the spectral-gap in the generator of the corresponding continuous-time Glauber dynamics
under an arbitrary fixed boundary condition $\tau$.
There exists an absolute constant $C>0$ (independent of $\Lambda,\tau$) such that $(\gap_\Lambda^\tau)^{-1} \leq n^C$.
\end{maintheorem}

Using well-known relations between $L^1$ and $L^2$ mixing (see, e.g.,~\cite{SaloffCoste})
an immediate corollary of our main result is that the total-variation mixing time of the Glauber dynamics is bounded above by a polynomial in the side-length $n$ (see Corollary~\ref{cor-tmix} in Section~\ref{subsec:prelim-mix-l1l2}).

Furthermore, we show a generalized version of the above theorem (see Theorem~\ref{thm-inverse-gap}) that applies to rectangles of unbounded aspect-ratio (e.g., an $n\times \mathrm{e}^n$ box), where the polynomial upper bound depends only on the shorter side-length.

Lower bounds on the inverse-gap can be derived already from the work of Onsager, as observed by Holley~\cite{Holley1}, demonstrating that it grows at least polynomially fast with $n$. Following this approach
we complement the above upper bound by showing that the inverse-gap is of order at least $n^{7/4}$ (see Theorem~\ref{thm-gap-lower-bound}).

\begin{remark*}
No attempt was made to optimize the exponent $C>0$ given in Theorem~\ref{thm-1}.
As we later elaborate, our proof determines this bound explicitly in terms of a
bound $c^+ < 1$ on the probability of a certain crossing event in the critical Fortuin-Kasteleyn representation of the Ising model (namely, the dependence on $1/(1-c^+)$ is logarithmic). The true exponent is believed to be universal and numerical experiments supporting this \cites{Grassberger,Ito,NB,WH,WHS} (among others) suggest that its value is about $2.17$.
\end{remark*}

\begin{figure}[t]
  {\,} \hfill
  \subfloat[Critical Ising model]{%
    \includegraphics[width=3in]{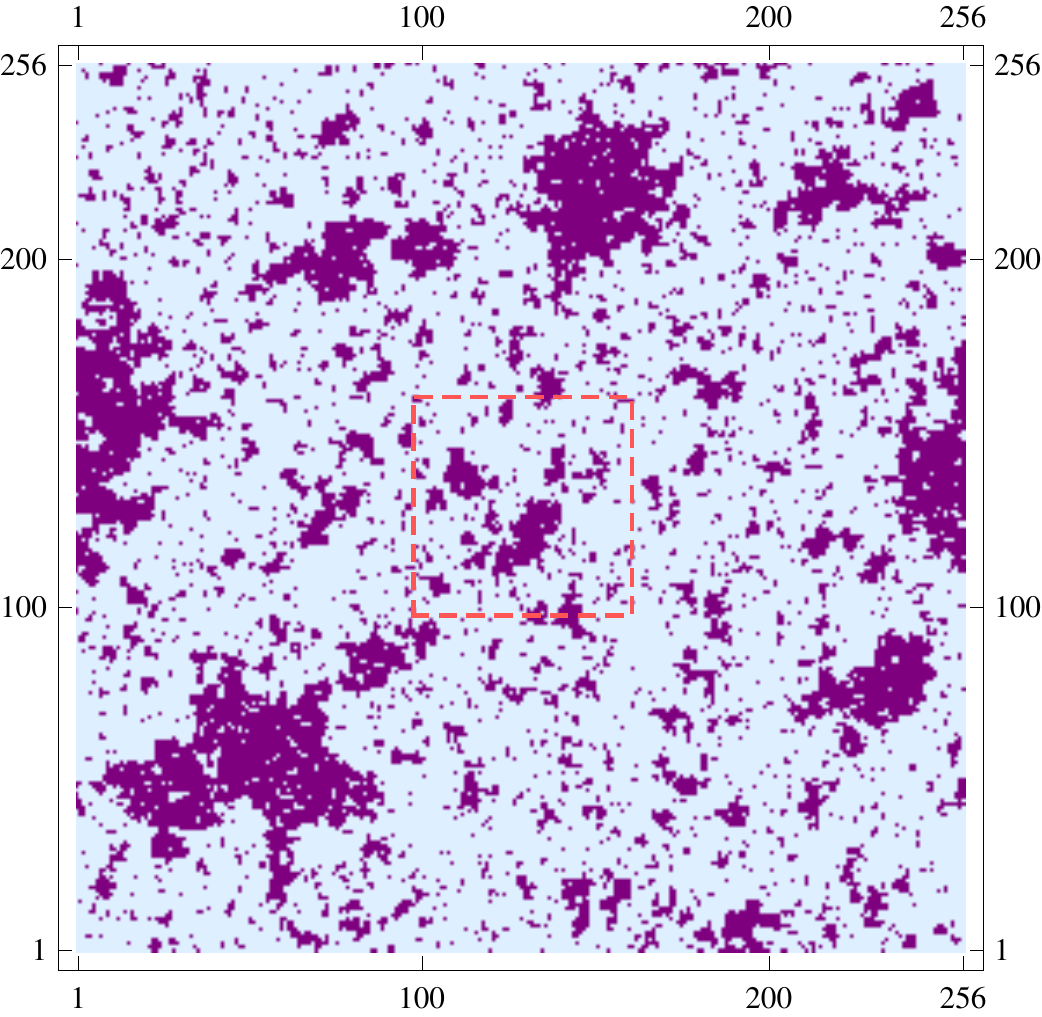}} \hfill
  \subfloat[Critical FK-Ising]{%
    \raisebox{0.5in}{\fbox{\includegraphics[width=1.75in]{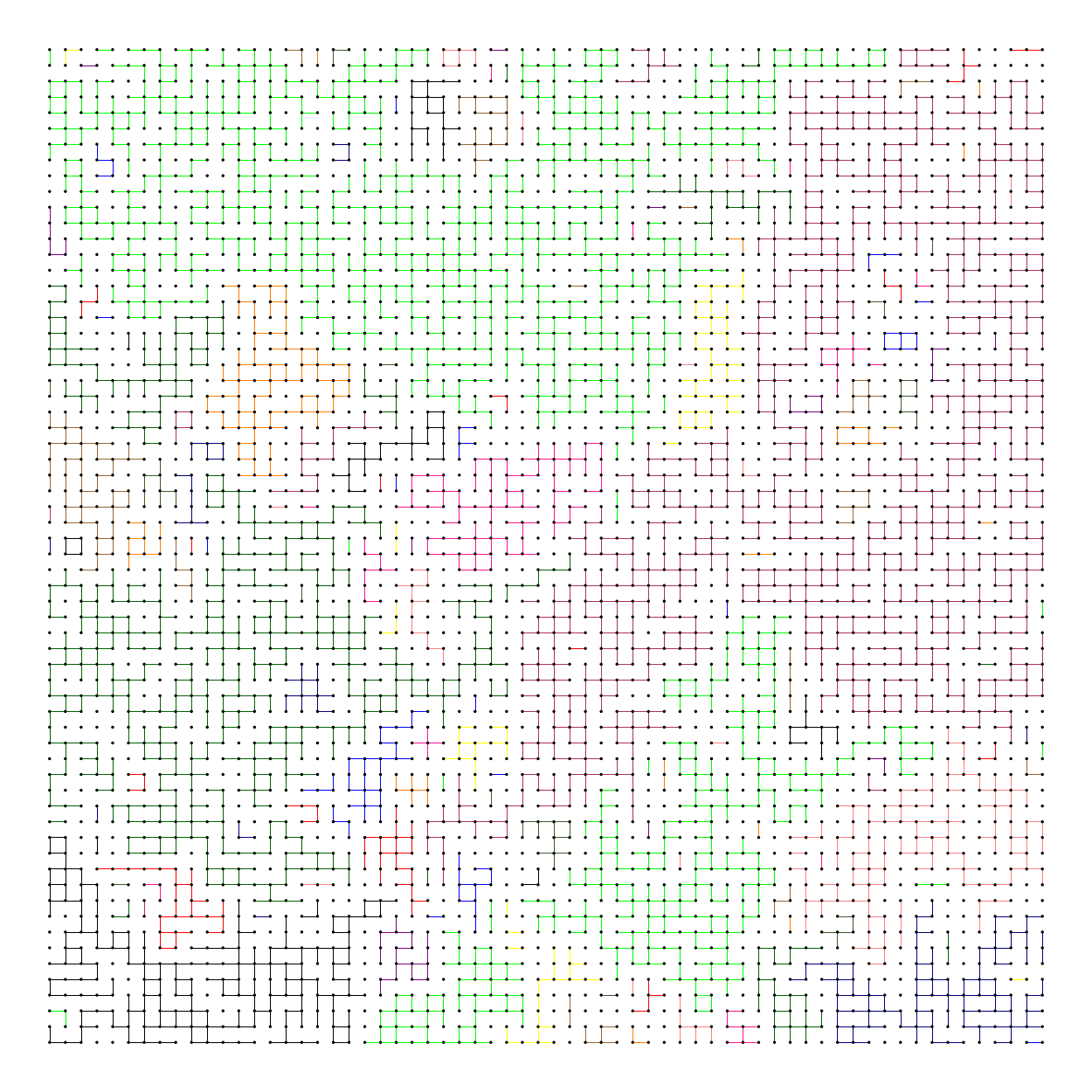}}}} \hfill
  {\,}
  \caption{The FK-representation for the critical Ising model.
  (A)~Critical Ising
  on a $256\times256$ lattice; $64\times64$ box highlighted.
  (B)~Coupled FK-configuration corresponding to highlighted box.}
\end{figure}

\subsection{Background and previous work}\label{sec:intro-previous}
In what follows we give a brief account of related work on the stochastic Ising model.
For a more extensive exposition, see e.g.~\cites{Liggett,Martinelli97}.

The classical Ising model on a finite box $\Lambda$ in the square lattice $\Z^2$ with no external field
is defined as follows. Its set of possible configurations is $\OmI=\{\pm1\}^\Lambda$, where each configuration
corresponds to an assignment of plus/minus spins to the sites in $\Lambda$. The probability that the system is in a
configuration $\sigma \in \OmI$ is given by the Gibbs distribution
\begin{equation}
  \label{eq-Ising}
  \mu_\Lambda(\sigma)  = \frac1{Z} \exp\left(\beta \sum_{u\sim v} \sigma(u)\sigma(v) \right) \,,
\end{equation}
where the partition function $Z$ is a normalizing constant.
The parameter $\beta$ is the inverse-temperature of the system; for $\beta \geq 0$ we say that the
model is ferromagnetic, otherwise it is anti-ferromagnetic.

These definitions extend to the infinite volume lattice $\Z^2$, where there is a critical point $\beta_c = \frac12\log(1+\sqrt{2})$ such that the Gibbs distribution is unique if and only if $\beta \leq \beta_c$.
We note that the focus of this work is restricted to the case of no external field ($h=0$) since otherwise there is no phase-transition in the above setting and the dynamics for the system is well understood.

Define the boundary of a set $\Lambda\subset V$, denoted by $\partial\Lambda$, as the neighboring sites of $\Lambda$ in $V\setminus\Lambda$ and call $\tau\in\{\pm1\}^{\partial \Lambda}$ a boundary condition. The Gibbs distribution conditioned on $\tau$ will be denoted by $\mu_\Lambda^\tau$, and put $\Lbar = \Lambda \cup \partial \Lambda$.
A periodic boundary condition on a box corresponds to a $2$-dimensional torus.

The Glauber dynamics for the Ising model is a family of continuous-time Markov chains on the state space $\OmI$, reversible with respect to the Gibbs distribution, given by the generator
\begin{equation}
  \label{eq-Glauber-gen}
  (\mathcal{L}f)(\sigma)=\sum_{x\in \Lambda} c(x,\sigma) \left(f(\sigma^x)-f(\sigma)\right)
\end{equation}
where $\sigma^x$ is the configuration $\sigma$ with the spin at $x$ flipped.
The transition rates $c(x,\sigma)$ are chosen to satisfy finite range interactions, detailed balance, positivity and boundedness and translation invariance (see Section~\ref{sec:prelim}).
As mentioned above, the two most notable examples for the choice of transition rates are
\begin{enumerate}[(i)]
\item \emph{Metropolis}: $  c(x,\sigma) = \exp\Big(2\beta\sigma(x)\sum_{y \sim x}\sigma(y)\Big)  \;\wedge\; 1\; $.
\item \emph{Heat-bath}:   $\;c(x,\sigma) = \bigg[1+ \exp\Big(-2\beta\sigma(x)\sum_{y \sim x}\sigma(y)\Big)\bigg]^{-1}\;$.
\end{enumerate}
These chains have intuitive and useful graphical interpretations: for instance, the heat-bath Glauber dynamics is equivalent to
updating the spins via i.i.d.\ rate-one Poisson clocks, each time resetting a spin and replacing it by a sample according to the conditional distribution given its neighbors.

Perhaps the most fundamental property of the dynamics is the gap in the spectrum of its generator, which governs the rate of convergence of the Glauber dynamics to equilibrium in $L^2(\mu)$.
 The spectral-gap is defined via the following Dirichlet form:
\begin{align}\label{eq-dirichlet-form}
\gap = \inf_f \frac{\mathcal{E}(f)}
{\var(f)}\,,
\end{align}
where the infimum is over all nonconstant $f\in L^2(\mu)$ and
\begin{align}
\mathcal{E}(f) &= \left<\mathcal{L} f,f\right>_{L^2(\mu)} = \frac12\sum_{\sigma,x} \mu(\sigma)c(x,\sigma)\left[f(\sigma^x)-f(\sigma)\right]^2\,.\label{eq-def-E(f)}
\end{align}
Alternatively, one may define the gap via the discrete-time analogue of the chain, in which case it is given in terms of the largest nontrivial eigenvalue of the transition kernel.


We next present a partial list of results obtained on the spectral-gap in various temperature regimes over the last half-century.

\subsubsection*{High temperature regime}
A series of breakthrough papers by Aizenman, Dobrushin, Holley, Shlosman, Stroock et al.\
(cf., e.g., \cites{AH,DobShl,Holley1,HoSt1,HoSt2,Liggett,LY,MO,MO2,MOS,SZ1,SZ2,SZ3,Zee1,Zee2}) beginning from the
late 1970's has developed the theory of the convergence rate of the Glauber dynamics to equilibrium.

Aizenman and Holley~\cite{AH} showed that the spectral gap of the dynamics on the infinite-volume lattice
is uniformly bounded whenever the Dobrushin-Shlosman uniqueness condition holds. Stroock and Zegarli{\'n}ski~\cites{Zee1,SZ1,SZ3}
 proved that the logartihmic-Sobolev constant is uniformly bounded given the Dobrushin-Shlosman mixing conditions (complete analyticity).
Finally, in 1994 Martinelli and Olivieri \cites{MO,MO2} extended this result to cubes under the more general condition of \emph{regular} complete analyticity, shown to hold for any $\beta<\beta_c$ in two dimensions.
 In particular, this confirmed the conjectured behavior of the inverse-gap in the high temperature regime in $\Z^2$.
This spectral-gap was shown to determine the precise asmyptotics of the $L^1$ mixing time of the dynamics in the recent work~\cite{LS}, where the authors established the cutoff phenomenon for the dynamics in this regime.

See the excellent surveys \cites{Martinelli97,Martinelli04} for further details.

\subsubsection*{Low temperature regime}
In this regime, in accordance with the conjectured critical slowdown behavior of the Ising model, the Glauber
dynamics under the free boundary condition was expected to converge exponentially slowly in the side-length.
This was first established by Schonmann~\cite{Schonmann} in 1987 for sufficiently low temperatures as a direct corollary from a large deviation estimate on the magnetization in a square. This large deviation result was extended to any $\beta > \beta_c$ by Chayes, Chayes and Schonmann~\cite{CCS}, implying that the inverse-gap under free boundary conditions is at least $\exp(-c n)$ for some $c=c(\beta) > 0$. This result, concurring with the projected behavior of the inverse-gap in this regime, also appeared explicitly and independently by Thomas~\cite{Thomas}.

In 1995, precise large deviation rate functions for the magnetization were established by Ioffe~\cite{Ioffe} for all $\beta > \beta_c$. Building on the work of Martinelli~\cite{Martinelli94} (treating low enough temperatures), this then culminated in remarkably sharp estimates for the inverse-gap throughout the low temperature regime in the work of Cesi, Guadagni, Martinelli and Schonmann~\cite{CGMS} in 1996.

\subsubsection*{Critical temperature}
Evidence that the inverse-gap is at least polynomial in the side-length at the critical temperature followed from
the polynomial decay in the spin-spin correlation, whose asymptotics were determined by Onsager~\cite{Onsager}.
On the other hand, despite the vast body of papers that accumulated on the non-critical regimes, no sub-exponential upper bound was known for the inverse-gap at criticality.

In fact, the only geometries where the inverse-gap of dynamics for the critical Ising model was shown to be polynomial (in the appropriate parameter for the given underlying geometry) were the complete graph~\cites{DLP1}, where the large symmetry renders the methods of proof useless for lattices, and the regular tree~\cite{DLP2}, whose non-amenability and non-transitivity also require geometry-specific methods. Even for these much simpler geometries the proofs were highly nontrivial.

Indeed, the complexity of the critical behavior can be witnessed by the scaling limit of the static Ising model. The understanding of the limit emerged in the last decade, pioneered by the introduction of the Schramm-Loewner Evolution (SLE) by Schramm~\cite{Schramm} (see~\cite{Werner1} for more on SLEs).
This powerful machinery has revolutionized the study of critical phenomena in two dimensions and allowed direct calculation of various critical exponents (see for instance \cites{LSW1,LSW2,Smirnov3} to name just a few).

Striking new results by Smirnov (cf.~\cites{Smirnov1,Smirnov2}) describe the full scaling limit of the Ising cluster interfaces at criticality as the Conformal Loop Ensemble (CLE) with parameter $\kappa=3$ (also see \cites{LW,Sheffield,Werner2}).

An important role in the development of the scaling limit theory for the Ising model was played by its counterpart, the Fortuin-Kasteleyn representation (formally defined in Section~\ref{sec:prelim}). Scaling limit results that were initially obtained for this model at criticality were thereafter translated to the corresponding Ising measure. For instance, the full ensemble of cluster interfaces in the FK-Ising model converges to a (nested) $\mathrm{CLE}_\kappa$ for $\kappa=16/3$.

Along side the ground breaking results on critical percolation models in two dimensions, new tools were developed to study conformally invariant systems. A recent application of this theory yielded Russo-Seymour-Welsh type estimates (\cites{Russo,SW}) for the crossing probability in a rectangle of bounded aspect ratio in the critical FK-Ising model under arbitrary boundary conditions, recently obtained by Duminil-Copin, Hongler and Nolin~\cite{DHN}.
See also~\cite{CS} for this result under a specific boundary condition, as well as~\cite{CN} for an argument inferring the general estimate from announced results on the convergence of the scaling limit of the spin cluster boundaries to $\mathrm{CLE}_3$ and from its Brownian loop soup representation.
As we later state, the RSW-estimate of~\cite{DHN} is a key ingredient in our proof.

As a consequence of Theorem~\ref{thm-1}, which verifies the conjectured polynomial behavior of the inverse-gap at criticality, the critical slowdown of the Ising model in 2 dimensions is now fully established.

\subsection{Polynomial mixing time under arbitrary boundary conditions}
Recall that Theorem~\ref{thm-1} bounded the inverse-gap for any fixed boundary condition.
Our proof in fact extends to free or periodic boundary conditions as well as mixed boundary conditions (see Theorem~\ref{thm-free-periodic}).
Furthermore, by a standard reduction the result also carries to the anti-ferromagnetic Ising model (see Corollary~\ref{cor-antiferro}).

To the best of our knowledge, this gives the first rigorous polynomial-time algorithm for approximately sampling the critical Ising model in $\Z^2$, also leading to an approximation of its partition-function, under \emph{arbitrary} (e.g.\ mixed) boundary conditions. (In the absence of boundary conditions Jerrum and Sinclair provided an efficient approximation scheme for the Ising partition-function on any graph in their celebrated work~\cite{JS}. Based on this algorithm Randall and Wilson~\cite{RW} gave an efficient approximate sampler for the Ising model on general graphs in 1999, applicable whenever the boundary conditions are free/all-plus/all-minus or periodic.)

\subsection{Efficient rigorous perfect simulation}
In their highly influential work~\cite{PW}, Propp and Wilson introduced the method
\emph{Perfect Simulation} (or Coupling From The Past) to exactly sample from the stationary measure of certain Markov chains. Perhaps their most prominent application for this method was sampling the Ising model at criticality via Glauber dynamics.

While the Propp-Wilson algorithm is guaranteed to produce a precise sample from the Ising measure,
it was not rigorously shown to be efficient, that is to have an expected running time polynomial in the size of the lattice. The only guarantee given for the running time was in terms of the total-variation mixing time, which was unknown at criticality. Note that experimental results found the method to be extremely effective in this regime (a fact that served as additional supporting evidence that the inverse-gap of the dynamics is polynomial at $\beta=\beta_c$).

By the new results in this work (namely, Corollary~\ref{cor-tmix}) we now have the first proof that the Propp-Wilson algorithm runs in polynomial time. This constitutes the first rigorously proven efficient algorithm (in addition to being very simple) for perfectly simulating the critical Ising model on the square lattice. Furthermore, this algorithm is valid under arbitrary boundary conditions.

\subsection{Main techniques}
A common feature of most analyses of the Glauber dynamics is to utilize the spatial-mixing properties of the static Ising model in order to control the mixing rate of the dynamical process.
In the high temperature regime, this is typically done by measuring the influence of individual boundary condition spins on sites and its rate of decay with distance. However, at the critical temperature, the slow decay of correlations precludes this approach. Indeed, as demonstrated by the spin-spin correlation result of Onsager and further illustrated by the existence with positive probability of ``large'' conformal loops as Ising cluster interfaces, there are long-range correlations at criticality which foil the standard coupling techniques.

Instead, we apply ideas from the study of the conformal invariance of the Ising model, and crucially the RSW-type estimate of~\cite{DHN}, to obtain the spatial-mixing result required for our analysis (Theorem~\ref{thm-spatial-mix}).
Rather than considering the effect of a single spin on the boundary, here we analyze the effect of an entire face of the boundary, deducing just enough spatial mixing to push our program through, with the help of
additional ingredients from the analysis of Markov chains, in particular the block-dynamics method~\cite{Martinelli97}.

\section{Preliminaries}\label{sec:prelim}
Throughout the paper we will use the notation $\llb x,y\rrb = [x,y]\cap \Z$, whereby for instance
a square lattice of side-length $n$ can be denoted by $\llb 1,n\rrb^2$.

\subsection{Glauber dynamics for the Ising model}
The Glauber dynamic for the Ising model on a finite box $\Lambda \subset \Z^2$, whose generator is given in \eqref{eq-Glauber-gen},
accepts any choice of transition rates $c(x,\sigma)$ which satisfy the following:
\begin{enumerate}[(1)]
\item \emph{Finite range interactions}: For some fixed $R>0$ and any $x \in \Lambda$, if $\sigma,\sigma' \in \OmI$ agree on the ball of diameter $R$ about $x$ then
$c(x,\sigma)=c(x,\sigma')$.
\item \emph{Detailed balance}: For all $\sigma\in \OmI$ and $x \in \Lambda$,
\[ \frac{c(x,\sigma)}{c(x,\sigma^x)} = \exp\Big(2\beta\sigma(x)\sum_{y \sim x}\sigma(y)\Big)\,.
\]
\item \emph{Positivity and boundedness}: The rates $c(x,\sigma)$ are uniformly bounded from below and above by some fixed $C_1,C_2 > 0$.
\item \emph{Translation invariance}: If $\sigma \equiv \sigma'(\cdot + \ell)$, where $\ell \in \Lambda$ and addition is according to the lattice metric,
then $c(x,\sigma) = c(x+\ell,\sigma')$ for all $x \in \Lambda$.
\end{enumerate}
The Glauber dynamics generator with such rates defines a unique Markov process, reversible with respect to the Gibbs measure $\mu_\Lambda^\tau$.

\subsection{Mixing in $L^1$ and in $L^2$}\label{subsec:prelim-mix-l1l2}
For any two distributions $\phi,\psi$ on $\Omega$, the \emph{total-variation distance} of $\phi$ and $\psi$ is defined as
\[\|\phi-\psi\|_\mathrm{TV} \deq \sup_{A \subset\Omega} \left|\phi(A) - \psi(A)\right| = \frac{1}{2}\sum_{x\in\Omega} |\phi(x)-\psi(x)|\,.\]
 The (worst-case) total-variation \emph{mixing-time} of an ergodic Markov chain $(X_t)$ with stationary distribution $\mu$, denoted by $\tmix=\tmix(1/\mathrm{e})$, is \[ \min\Big\{t : \max_{x \in \Omega} \| \P_x(X_t \in \cdot)- \mu\|_\mathrm{TV} \leq 1/\mathrm{e} \Big\}\,,\]
where $\P_x$ denotes the probability given that $X_0=x$. It is easy and well known (cf., e.g., \cite{SaloffCoste})
that $\tmix \leq \gap^{-1} \log\frac{\mathrm{e}}{\mu_{\min}}$, where $\mu_{\min}=\min_{x\in\Omega} \mu(x)$.
Together with Theorem~\ref{thm-1}, this immediately implies the following:
\begin{corollary}\label{cor-tmix}
Consider the critical Ising model on a finite box $\Lambda \subset \Z^2$ of side-length $n$, i.e.\ at inverse-temperature $\beta_c=\frac12\log(1+\sqrt{2})$.
Let $\tmix^\tau$ be the worst-case total-variation mixing time of the corresponding continuous-time Glauber dynamics
under an arbitrary fixed boundary condition $\tau$.
There exists an absolute constant $C>0$ (independent of $\Lambda,\tau$) such that $\tmix^\tau \leq n^C$.
\end{corollary}

\subsection{Single-site vs.\ Block dynamics}
Let $\cB= \{B_1,\ldots,B_k\}$ where the $B_i$'s are subsets of sites to be referred to as ``blocks''.
The (continuous-time) \emph{block dynamics} corresponding to $\cB$ is the following Markov chain: Each
block corresponds to a rate-one Poisson clock and upon it ringing we update the block according to the stationary distribution given the rest of the system. That is,
the entire set of spins of the chosen block is updated simultaneously, whereas all other spins remain unchanged.
One can verify that this dynamics is reversible with respect to the Gibbs distribution.

The following proposition reveals the remarkable connection between the single-site dynamics and the block dynamics.
\begin{proposition}[\cite{Martinelli97}*{Proposition 3.4}]\label{prop-block-single}
Consider the Glauber dynamics for the Ising model on $\Lambda\subset \Z^2$ with boundary condition $\tau$.
Let $\gap_\Lambda^\tau$ be the spectral-gap of the single-site
dynamics on $\Lambda$ and $\gap_\mathcal{B}^\tau$ be the spectral-gap of the block dynamics corresponding to $B_1,\ldots,B_k$, an arbitrary cover of $\Lambda$.
 The following holds:
\begin{equation*}
\gap_\Lambda^\tau \geq
\frac{\sum_{\sigma} \mu_\Lambda^\tau(\sigma)\sum_{x\in\Lambda} c(x,\sigma)[f(\sigma^x)-f(\sigma)]^2}
{\sum_{\sigma} \mu_\Lambda^\tau(\sigma)\sum_{x\in\Lambda} N_x c(x,\sigma)[f(\sigma^x)-f(\sigma)]^2}
\gap_\mathcal{B}^\tau \inf_i \inf_{\varphi} \gap_{B_i}^\varphi
\,,
\end{equation*}
where $N_x = \#\{i:B_i \ni x\}$. In particular,
\begin{equation*}
  \gap_\Lambda^\tau \geq \Big(\sup_{x\in \Lambda} \#\{i:B_i \ni x\}\Big)^{-1} \gap_\mathcal{B}^\tau \inf_i \inf_{\varphi} \gap_{B_i}^\varphi \,.
\end{equation*}
\end{proposition}
A notable example for key properties of the Glauber dynamics that were established using this approach is the
estimate on the spectral-gap throughout the high temperature regime (cf., e.g., \cite{Martinelli97}).

\subsection{FK model}
Invented by Fortuin and Kasteleyn~\cite{FK} around 1969, the FK-model (also known as the \emph{random-cluster} model)
with parameter $q$ for a graph $\Lambda$ with edge set $E$ is defined as follows.
Its state space is all assignments of open/closed values to the edges (resp. $1$ and $0$ values),
where the probability of a configuration $\omega\in\{0,1\}^E$ is given by the FK-measure:
\[ \nu_\Lambda(\omega) = \frac1{Z} p^{o(\omega)}(1-p)^{c(\omega)}q^{k(\omega)}\,,\]
where $o(\omega),c(\omega)$ denote the number of open and closed edges in $\omega$ resp., $k(\omega)$ counts the number of clusters and $Z$ is a normalizing constant.

The case $q=2$ is closely related to the ferromagnetic Ising measure, as detailed by the following coupling due to Edwards and Sokal~\cite{ES}, hence this special case is also known as FK-Ising (or the FK-representation of the Ising model).
Let $\OmFK$ and $\OmI$ denote the state spaces of the Ising and FK measures respectively; the joint probability of $\sigma\in\OmI$ and $\omega\in\OmFK$ according to the coupling is
\[ \P(\sigma,\omega) \propto \prod_{e\in E} \left[(1-p)\one_{\{\omega(e)=0\}} + p\one_{\{\omega(e)=1\}}\one_{\{\sigma(x)=\sigma(y)\}}\right]\,,\]
where the relation between $0<p<1$ and $\beta>0$ is given by
\[ p = 1-\exp(-2\beta)\]
(note that this sometimes appears as $p=1-\exp(-\beta)$, e.g.\ in the context of Potts measures, resulting from a slightly different normalization of the Ising Hamiltonian). Unless stated otherwise, set $\beta$ and $p$ to their critical values \[\beta=\beta_c=\frac12\log(1+\sqrt{2})~,~p=p_c=p_\mathrm{sd} = \frac{\sqrt{2}}{1+\sqrt{2}}\,.\]
A useful corollary of the coupling is that an Ising configuration $\sigma \sim \mu_\Lambda$ can be obtained from the FK-Ising representation $\omega\sim\nu_\Lambda$ by
assigning i.i.d.\ uniform spins to the clusters.

One can also consider the FK-measure under boundary conditions. Here the constraints are ``wiring'' of sites,
thereby affecting the cluster-structure of $\omega$ (wired sites are part of the same cluster). Given a wiring $\tau$ we denote by $\nu^\tau_\Lambda$ the FK-measure with respect to this boundary condition.
The two extreme boundary conditions are the ``wired'' $\tau=1$ (all boundary vertices are pair-wise connected) and the ``free'' $\tau=0$ (no wiring) boundary conditions.

The FK-model is monotone with respect to boundary conditions: In the natural partial order on the wirings, whenever $\xi\leq \eta$ the measure $\nu_\Lambda^\eta$ stochastically dominates $\nu_\Lambda^\xi$.
This property is particularly useful when combined with the so-called \emph{Domain Markov} property, asserting that
$ \nu^\xi_\Lambda\left(\omega\in\cdot\mid \omega(\Delta)\right)$ for some subset of the edges $\Delta\subset E$ is
precisely the FK-measure on the remaining (unexposed) edges under the boundary condition induced by
the wirings of $\xi$ and $\omega(\Delta)$.

A final ingredient which we will need is the FKG-inequality, stated next for the special case
of the FK-measure:
\begin{theorem}[\cites{FKG,Holley2}, special case]
For any graph $\Lambda$ and boundary condition $\xi$, if $X,Y:\OmFK\to\R$
are two increasing (decreasing) functions then
\[ \nu_\Lambda^\xi(X Y) \geq\nu_\Lambda^\xi(X) \nu_\Lambda^\xi(Y)\,.\]
\end{theorem}
Throughout the paper we will use the term \emph{cluster} for its usual interpretation as a maximal subset of sites connected by open paths, including via the boundary wiring. As a distinction, we will use the term \emph{component} to refer to a maximal subset of sites that are connected by open paths without using the boundary wiring.

See e.g.~\cite{Grimmett} and~\cite{Alexander} for further information on the FK-model.

\section{Polynomial mixing at criticality}

In this section we provide the proof of Theorem~\ref{thm-1}, establishing that the inverse-gap of the Glauber dynamics for the critical Ising model on a finite square lattice is polynomial in the side-length.
We will in fact prove the following stronger version of Theorem~\ref{thm-1} which allows for rectangles with unbounded aspect-ratio and bounds the inverse-gap solely in terms of their shorter side-length.
\begin{theorem}\label{thm-inverse-gap}
Consider the critical Ising model on a box $\Lambda \subset \Z^2$ with dimensions $m\times n$.
Let $\gap_\Lambda^\tau$ be the spectral-gap in the generator of the corresponding continuous-time Glauber dynamics
under an arbitrary fixed boundary condition $\tau$.
There exists an absolute constant $C>0$ (independent of $\Lambda,\tau$) such that for any $m=m(n)$ we have $(\gap_\Lambda^\tau)^{-1} \leq n^C$.
\end{theorem}
For instance, the inverse-gap for the critical Ising model on an extremely long rectangle, e.g.\ an $n\times \exp(\exp(n))$ box, is bounded by the exact same $n^C$ bound given for the square.

\subsection{A spatial-mixing result and proof of Theorem~\ref{thm-inverse-gap}}
We derive our bound on the inverse-gap from the following spatial-mixing result which is of independent interest.  Its proof is given in Section~\ref{s:spatialProof}.
\begin{theorem}\label{thm-spatial-mix}
Let $\Lambda = \llb 1,r \rrb \times \llb 1 , r'\rrb$ for some integers $r,r'$ satisfying $r'/r > \alpha > 0$ with $\alpha$ fixed and let
$\Lambda_\btop = \llb 1,r \rrb \times \llb \rho r ,r' \rrb$ for some $\rho$ satisfying $\alpha \leq \rho<r'/r$.
Let $\xi,\eta$ be two boundary conditions on $\Lambda$ that differ only on the bottom boundary $\llb 1,r\rrb \times \{0\}$.
Then \[ \left\| \mu^\xi_\Lambda(\sigma(\Lambda_\btop)\in\cdot) - \mu^\eta_\Lambda(\sigma(\Lambda_\btop)\in\cdot) \right\|_\tv \leq \exp(-\delta \rho)\,,\]
where $\delta > 0$ is a constant that depends only on $\alpha$.
\end{theorem}

\begin{proof}
  [\emph{\textbf{Proof of Theorem~\ref{thm-1}}}]
We proceed by using the method of Block Dynamics to recursively relate the spectral-gaps of boxes of decreasing sizes, following the approach of Martinelli~\cite{Martinelli97}.

Let $\Lambda = \llb 1,r \rrb \times \llb 1 , r'\rrb$ be a box of dimensions $r \times r'$ and without loss of generality assume $r' \geq r$. For some $\ell \in \{1,\ldots,\lfloor \sqrt{r'/r}\rfloor\}$ partition $\Lambda$ into two vertically overlapping boxes:
\begin{align*}
\Lambda_1 &= \Lambda_1(\ell)= \llb 1,r \rrb \times \llb \tfrac13 r' + \tfrac{\ell-1}3\sqrt{rr'} , r'\rrb \\
 \Lambda_2&= \Lambda_2(\ell)=\llb 1,r \rrb \times \llb 1 , \tfrac13 r' + \tfrac{\ell}3\sqrt{rr'}\rrb
\end{align*}
(e.g., if $r'=r$ then these overlap in the middle third of $\Lambda$).

Denote by $\gap_\cB^\xi$ the spectral-gap of the block-dynamics on $\Lambda$ (see Section~\ref{sec:prelim} for the definition of this dynamics) corresponding to the blocks $\cB=\{\Lambda_1,\Lambda_2\}$ and the boundary condition $\xi$.
\begin{claim}\label{clm-block-dyn}
For any boundary condition $\xi$ on the box $\Lambda$, the spectral-gap of the
  block-dynamics that corresponds to $\cB = \{ \Lambda_1,\Lambda_2 \}$ as defined above satisfies $\gap^\xi_\cB \geq 1-\exp(-c\sqrt{r'/r})$, where $c>0$ is an absolute constant.
\end{claim}
\begin{proof}
It is sufficent to consider the discrete-time version of the dynamics, whose spectral-gap we will denote by $\bar{\gap} =\frac12 \gap_\cB^\xi$. Let $(X_t)$ and $(Y_t)$ be two instances of the discrete time block-dynamics on $\Lambda$ under the boundary condition $\xi$, started at two arbitrary initial positions $X_0$ and $Y_0$ respectively.

Couple the choices of updated blocks in $(X_t)$ and $(Y_t)$ and consider the case where two distinct blocks are updated at times $t,t+1$ for some $t\geq 1$. If we first update $\Lambda_1$ then upon updating this block the boundary condition is the same given $X_{t-1},Y_{t-1}$ except (possibly) on the boundary with $\Lambda\setminus \Lambda_1$.  Using a maximal coupling we have by Theorem~\ref{thm-spatial-mix} (setting parameters $\alpha=\frac13$ and $\rho=\frac13\sqrt{r'/r}$) that with probability at least $1-\exp(-\delta\rho)$ we can couple $X_t$ and $Y_t$ so that they agree on $\llb 1,r \rrb \times \llb \frac13 r' + \frac{\ell}3\sqrt{rr'} , r'\rrb$. In particular they would agree on $\Lambda\setminus\Lambda_2$.
On this event, upon updating $\Lambda_2$ in the following step the two chains may be coupled with probability $1$.
Similarly, if we first update $\Lambda_2$ then for the same $\alpha,\rho$ we have a probability of at least $1-\exp(-\delta\rho)$ to couple $X_t,Y_t$ so that they agree on $\llb 1,r \rrb \times \llb 1, \frac13 r' + \frac{\ell-1}3\sqrt{rr'}\rrb$, again implying coalescence of $X_{t+1},Y_{t+1}$.
We deduce that for any $t \geq 2$
\[ \P(X_t \neq Y_t) \leq \bigg(1-\frac12\left(1-\mathrm{e}^{-\delta\rho}\right)\bigg)^{t-1} = \bigg(\frac12\left(1+\mathrm{e}^{-(\delta/3)\sqrt{r'/r}}\right)\bigg)^{t-1}\,,\]
and in particular the exponential decay of the coupling time satisfies
\begin{align}\label{eq-coupling-decay} \lim_{t\to\infty}\frac1{t}\log \max_{x,y} \|\P_x(X_t \in \cdot)-\P_y(X_t\in\cdot)\|_\mathrm{TV} \leq
\frac12\left(1+\mathrm{e}^{-(\delta/3)\sqrt{r'/r}}\right).\end{align}
As the exponential decay of $(X_t)$ to equilibrium is governed by $\lambda$, the second largest (in absolute value) eigenvalue of the transition kernel,
we have that $\lambda$ is precisely the above limit, hence
\[\gap_{\cB}^\xi = 2 \bar{\gap} = 2(1-\lambda) \geq 1-\mathrm{e}^{-(\delta/3)\sqrt{r'/r}}\,.\]
This completes the claim.
\end{proof}


To obtain Theorem~\ref{thm-inverse-gap} we now relate the spectral-gap of $\Lambda$ with that of the $\Lambda_i$'s. By the above claim and Proposition~\ref{prop-block-single} we obtain that for some absolute $c>0$ and any choice of $\ell$ we have
\begin{equation*}
(\gap_\Lambda^\xi)^{-1} \leq \frac{\sum_{x\in\Lambda} N_x(\ell) A^\xi_{x}}
{\sum_{x\in\Lambda} A^\xi_{x}}
\left(1-\mathrm{e}^{-c\sqrt{r'/r}}\right)^{-1} \max_{i, \eta} (\gap_{\Lambda_i(\ell)}^\eta)^{-1}
\,,
\end{equation*}
where $A^\xi_{x} = \sum_{\sigma} \mu_\Lambda^\xi(\sigma) c(x,\sigma)[f(\sigma^x)-f(\sigma)]^2$
and $N_x(\ell) = \#\{i:\Lambda_i(\ell) \ni x\}$. Averaging over the $L=\lfloor\sqrt{r'/r}\rfloor$ admissible values of $\ell$ now yields
\begin{align*}
(\gap_\Lambda^\xi)^{-1} &\leq \frac1{L} \sum_{\ell=1}^L \frac{\sum_{x\in\Lambda} N_x(\ell) A^\xi_{x}}
{\sum_{x\in\Lambda} A^\xi_{x}}
\left(1-\mathrm{e}^{-c\sqrt{r'/r}}\right)^{-1} \max_{i,\eta} (\gap_{\Lambda_i(\ell)}^\eta)^{-1}\\
&\leq \bigg(\max_{x\in\Lambda} \sum_{\ell=1}^L \frac{N_x(\ell)}{L}\bigg)\left(1-\mathrm{e}^{-c\sqrt{r'/r}}\right)^{-1} \max_{i,\ell,\eta} (\gap_{\Lambda_i(\ell)}^\eta)^{-1}\\
&\leq \bigg(1+\frac1{\lfloor \sqrt{r'/r}\rfloor}\bigg)\left(1-\mathrm{e}^{-c\sqrt{r'/r}}\right)^{-1} \max_{i,\ell,\eta} (\gap_{\Lambda_i(\ell)}^\eta)^{-1}\,,
\end{align*}
where the last inequality is due to the fact that (crucially)
every $x\in\Lambda$ appears in the overlap $\Lambda_1(\ell)\cap\Lambda_2(\ell)$ for at most one value of $\ell$.
Observe that since each $\Lambda_i(\ell)$ has dimensions $r \times r''$ for some $\frac13 r' \leq r'' \leq \frac23 r'$ we can write
\begin{align}\label{eq-recursion-1}
(\gap_\Lambda^\xi)^{-1} &\leq \bigg(1+\frac1{\lfloor \sqrt{r'/r}\rfloor}\bigg)\left(1-\mathrm{e}^{-c\sqrt{r'/r}}\right)^{-1}
\max_{\Lambda'} \max_{\eta} (\gap_{\Lambda'}^\eta)^{-1}\,,
\end{align}
where $\Lambda'$ runs over all boxes of the form
$\llb 1,r\rrb \times \llb 1,r''\rrb$ with $\frac13 r' \leq r'' \leq \frac23 r'$.
Therefore, as long as $r'' \geq r$ we can recursively apply
the above argument and obtain that
\begin{align*}
(\gap_\Lambda^\xi)^{-1} &\leq \vartheta_0 \max_{\Lambda'} \max_{\eta} (\gap_{\Lambda'}^\eta)^{-1}\,,
\end{align*}
where $\Lambda'$ goes over all boxes of the form
$\llb 1,r\rrb \times \llb 1,r''\rrb$ with $\frac13 r \leq r'' \leq r$ and where
\[ \vartheta_0 = \prod_{k=0}^{\infty} \left(1+(3/2)^{-k/2}\right) \left(1-\mathrm{e}^{-c (3/2)^{k/2}}\right)^{-1} < \infty\,.\]
An additional application of the recursion \eqref{eq-recursion-1} (this time reversing the roles of $r,r''$) now yields
\begin{align*}
(\gap_\Lambda^\xi)^{-1} &\leq \vartheta \max_{\Lambda'} \max_{\eta} (\gap_{\Lambda'}^\eta)^{-1}\,,
\end{align*}
where $\vartheta = 2\vartheta_0/(1-\mathrm{e}^{-c})$ and $\Lambda'$ runs over all boxes of the form $\llb 1,a\rrb\times\llb1,b\rrb$
with $a \wedge b \leq \frac23 r$.

The proof of Theorem~\ref{thm-inverse-gap} is now concluded by repeatedly applying this argument $\log_{3/2} n$ times inductively, implying that the inverse-gap of the single site dynamics on an $m\times n$ square lattice with arbitrary boundary conditions
is at most $\vartheta^{\log_{3/2} n} = n^{\log_{3/2}\vartheta}$,
as required.
\end{proof}

\subsection{Proof of Theorem~\ref{thm-spatial-mix} (spatial mixing for critical Ising on $\Z^2$)}\label{s:spatialProof}
The following is an immediate corollary of the Edwards-Sokal~\cite{ES} coupling of the Ising and FK-Ising models under free boundary conditions:
\begin{fact}\label{fact-coupling}
Let $\Lambda\subset \Z^2$ be a finite box and let $\xi$ be a wiring of its boundary vertices $\partial \Lambda$.
Let $S_1,S_2,\ldots$ denote the clusters of $\partial \Lambda$ that are induced by the boundary condition $\xi$.
Let $\omega\in\OmFK(\Lbar)$ be distributed according to the FK-Ising measure $\nu_{\Lbar}^\xi$ and transform it to a spin configuration $\sigma\in\OmI(\Lbar)$ by selecting an i.i.d.\ spin per cluster. Then $\sigma$ is distributed according to $\mu_{\Lbar}$ conditioned on every $S_i$ receiving the same spin in each of its sites.
\end{fact}
Indeed, to see the above simply observe that wiring two given sites $u,v$ is equivalent via the Edwards-Sokal coupling to introducing an interaction of strength $\beta(u,v) =\infty$ between these sites in the corresponding Ising model, i.e.\ conditioning that they obtain the same spin (see Fig.~\ref{fig:fk-ising-bc}).

\begin{figure}[t]
  {\,} \hfill
  \subfloat[Ising model]{%
    \fbox{\includegraphics[width=2in]{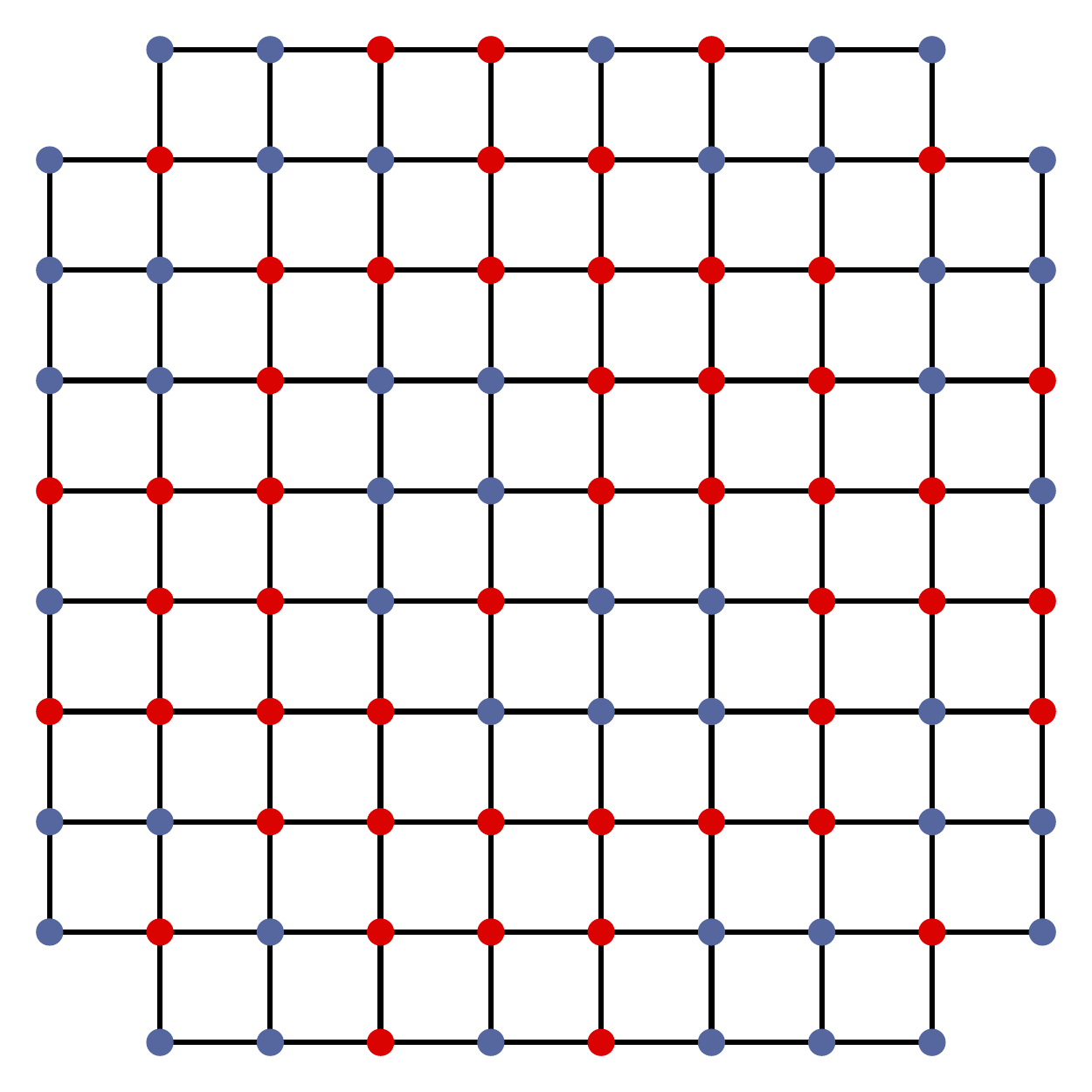}}} \hfill
  \subfloat[FK-Ising model]{%
    \fbox{\includegraphics[width=2in]{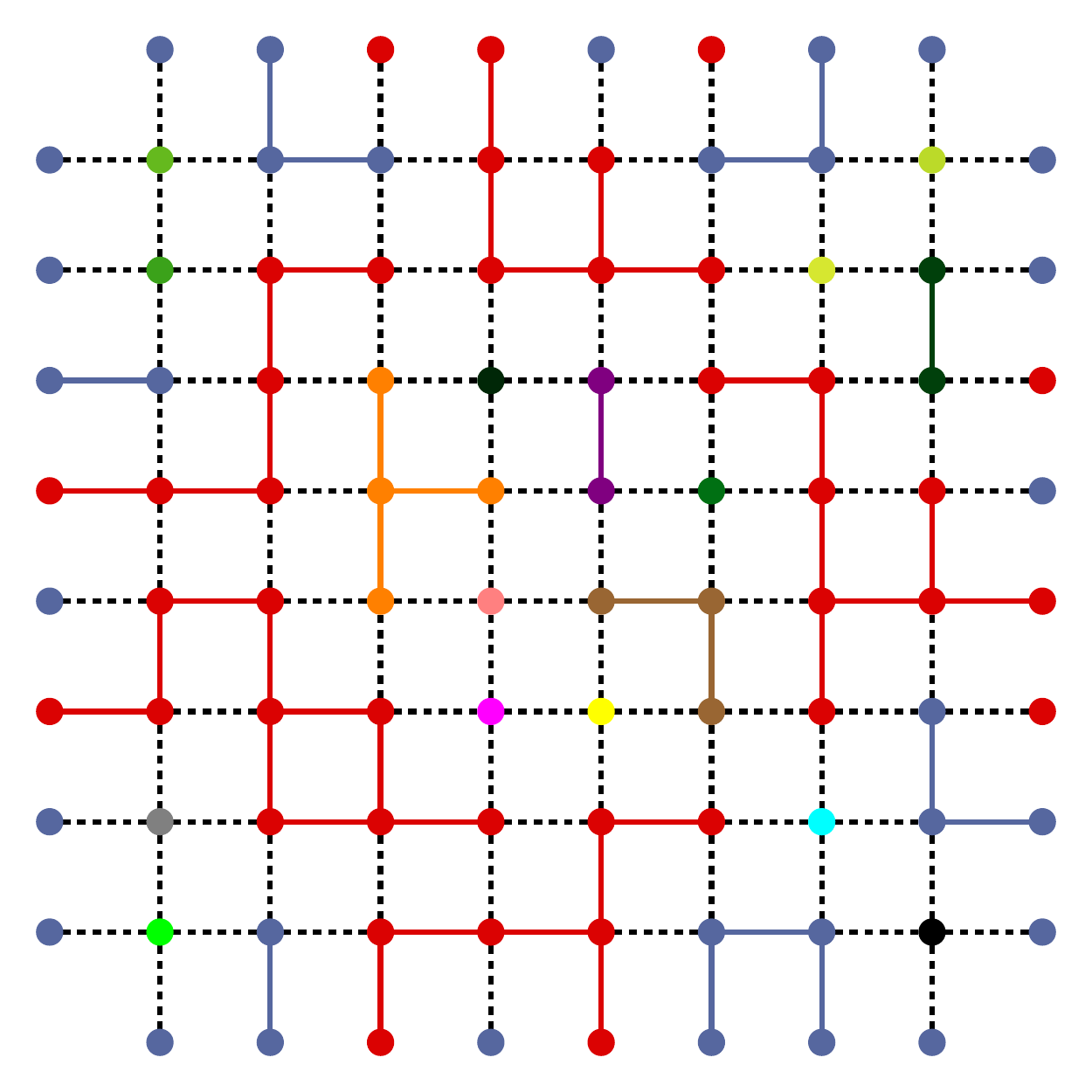}}} \hfill
  {\,}
  \caption{Coupling of the Ising model and its FK-Ising representation under arbitrary fixed boundary.
Blue sites and red sites are each wired pair-wise in the FK-configuration on the right.}
\label{fig:fk-ising-bc}
\end{figure}

An arbitrary boundary condition $\xi$ for the Ising model on $\Lambda$ induces a partition of the vertices of $\partial \Lambda$
 into two subsets, $\cP(\xi),\cM(\xi)$, corresponding to the plus-spins and minus-spins respectively.
With Fact~\ref{fact-coupling} in mind, by slight abuse of notation we will also let $\xi$ denote the boundary condition for the FK-model generated by wiring all the vertices of $\cP$ amongst themselves pair-wise and doing the same for those in $\cM$. Further define the following event for configurations $\omega\in \OmFK(\Lbar)$:
\[ A^\xi \deq \left\{ \cP(\xi) \not\rightsquigarrow\cM(\xi) \right\}\,,\]
that is, $\omega$ does not admit an open path that connects the two subsets $\cP,\cM$.
By the above fact, if we condition that a configuration $\omega\sim\nu^\xi_{\Lbar}$ satisfies the event $A^\xi$ then the Edwards-Sokal procedure for generating an Ising configuration from it agrees with the Ising boundary condition $\xi$ with probability $\frac14$ independently of $\omega$. We thus arrive at the following corollary:
\begin{corollary}\label{cor-FK-Ising-bc}
  Let $\xi$ be an arbitrary boundary condition for the Ising model on $\Lambda$ and define $\cP(\xi),\cM(\xi)$ and $A^\xi$ as above. Let $\omega \in \OmFK$ be distributed according to $\nu^\xi_{\Lbar}(\cdot \mid A^\xi)$ and produce an Ising configuration $\sigma\in \OmI$ from $\omega$ by assigning a plus-spin to the cluster of $\cP$, a minus-spin to the cluster of $\cM$ and i.i.d.\ uniform spins to all other clusters. Then $\sigma \sim \mu^\xi_\Lambda$.
\end{corollary}
Note that for some boundary conditions $\xi$ the probability of $A^\xi$ might be extremely small, e.g.\ exponentially small in the side-length of $\Lambda$, in light of which  treating this conditional space is quite delicate.

Let $\xi,\eta$ be the Ising boundary conditions given in Theorem~\ref{thm-spatial-mix} and define
\begin{align*}
  \Gamma \deq \llb 0,r+1\rrb \times \{0\}\,,
\end{align*}
while recalling that by definition $\xi,\eta$ differ at most on the sites of $\Gamma$. In light of Corollary~\ref{cor-FK-Ising-bc}, consider the FK-Ising model on $\Lbar = \llb 0,r+1\rrb \times \llb 0 ,r'+1\rrb$
under the two wirings corresponding to $\xi,\eta$ as defined above (wiring the sites in $\cP$ pairwise and the sites in  $\cM$ pairwise).

We next wish to construct a coupling of the measures \[\left\{\nu_{\Lbar}^\xi(\cdot \mid A^\xi)~,~ \nu_{\Lbar}^\eta(\cdot \mid A^\eta)~,~ \nu_{\Lbar}^1\right\}\,,\] where $\nu_{\Lbar}^1$ is the FK-measure under the (fully) wired boundary condition.
This will be achieved by gradually exposing all the sites of the configuration $\omega^1 \sim \nu_{\Lbar}^1$ that are connected by an open path to the bottom boundary $\Gamma$. Let $\Xi=\Xi(\omega^1)$ denote this set of sites and consider the following process for revealing it:
\begin{itemize}
  \item Initialize $\Xi$ to consist of the sites of $\Gamma$.
  \item Order the edges in $\Lambda$ arbitrarily $\{e_1,e_2,\ldots\}$.
  \item Repeatedly reveal the smallest (as per the above ordering) unexposed edge $e$ that is incident to $\Xi$, setting its value in $\omega^1$ via an independent unit variable $U_e$. If open, add the new endpoint of $e$ to $\Xi$.
  \item Let $T$ be the minimal time after which $\Xi$ has been exhausted (that is, the number of edges exposed in the above procedure).
\end{itemize}
Upon exposing the sites of $\Xi$ we construct the configurations $\omega^\xi \sim \nu_{\Lbar}^\xi(\cdot \mid A^\xi)$ and
$ \omega^\eta \sim \nu_{\Lbar}^\eta(\cdot \mid A^\eta)$ as follows.
Let $\Delta_t$ denote the set of edges exposed up to (including) time $t$, and suppose that at time $t+1$ we are about to expose the edge $e$ via the unit variable $U_e$. By definition,
\[ \omega^1(e) = \one \left\{ U_e \leq \nu_{\Lbar}^1\left( \omega(e)=1 \given \omega(\Delta_{t}) = \omega^1(\Delta_{t})\right) \right\}\,,\]
and we determine $\omega^\tau(e)$ for $\tau \in \{\xi,\eta\}$ analogously using the same $U_e$:
\[ \omega^\tau(e) = \one \left\{ U_e \leq \nu_{\Lbar}^\tau\left( \omega(e)=1 \given A^\tau\,,\, \omega(\Delta_{t}) = \omega^\tau(\Delta_{t})\right) \right\}\,.\]
Beyond time $T$, we reveal the remaining edges according to their aforementioned ordering while following the same recipe.
Clearly this construction has the correct marginals $\nu_{\Lbar}^\xi(\cdot\mid A^\xi),\nu_{\Lbar}^\eta(\cdot\mid A^\eta),\nu_{\Lbar}^1$ and the following claim establishes that $\omega^1$ dominates $\omega^\xi,\omega^\eta$.

\begin{claim}\label{clm-dom}
For $\tau\in\{\xi,\eta\}$ and any integer $t$ the above-defined coupling satisfies $\omega^\tau(\Delta_t) \leq \omega^1(\Delta_t)$.
\end{claim}
\begin{proof}
Let $\tau \in \{\xi,\eta\}$, assume by induction that the statement of the claim holds for $t$ and let $e$ be the edge exposed at time $t+1$.
As this means that $\omega^\tau (\Delta_{t}) \leq \omega^1(\Delta_{t})$, the Domain Markov property and monotonicity of boundary conditions give
\begin{align*}
 \nu_{\Lbar}^1\left( \omega(e) = 0 \given\omega(\Delta_{t}) = \omega^1(\Delta_{t}) \right) &\leq
\nu_{\Lbar}^\tau\left( \omega(e) = 0 \given \omega(\Delta_{t}) = \omega^\tau(\Delta_{t}) \right)\,.
\end{align*}
Since $A^\tau$ is a decreasing event, the FKG-inequality for the FK-Ising model together with the Domain Markov property (imposing a boundary condition that arises both from $\tau$ and from $\omega^\tau(\Delta_t)$) ensures that
\begin{align*} \nu_{\Lbar}^\tau\left(\omega^\tau(e)=0\,,\,A^\xi \given \omega(\Delta_t)=\omega^\tau(\Delta_t)\right) \geq
& \nu_{\Lbar}^\tau\left(\omega^\tau(e)=0 \given \omega(\Delta_t)=\omega^\tau(\Delta_t)\right)\\
\cdot~ & \nu_{\Lbar}^\tau\left(A^\xi \given \omega(\Delta_t)=\omega^\tau(\Delta_t)\right)\,,
\end{align*}
and combining the last two equations we deduce that
\begin{align*}
\nu_{\Lbar}^1\left( \omega(e) = 0 \given\omega(\Delta_{t}) = \omega^1(\Delta_{t}) \right) &\leq \nu_{\Lbar}^\tau\left( \omega(e) = 0 \given A^\tau\,,\,\omega(\Delta_{t}) = \omega^\tau(\Delta_{t}) \right)\,.
\end{align*}
The statement now follows from the definition of the coupling.
\end{proof}

Observe that if $\Xi$ is not incident to the top boundary of $\Lambda$ then necessarily the edge-boundary between $\Xi$ and its complement $\Xi^c$ is entirely closed in $\omega^1$, and in particular there is a horizontal crossing path in the dual. By the monotonicity argued in Claim~\ref{clm-dom} this carries to $\omega^\tau$ for $\tau\in\{\xi,\eta\}$ as well.

The following RSW-type estimate of \cite{DHN} for the critical FK-model under arbitrary boundary condition will now imply
that with positive probability $\Xi$ is confined to the box $\llb 0,r+1 \rrb \times \llb 0,\rho r-1\rrb$ (as illustrated in Fig.~\ref{fig:fkclust}):
\begin{theorem}[\cite{DHN}*{Theorem~1}]\label{thm-RSW}
  Let $0 < \alpha_1 < \alpha_2$ . There exist two  constants $0 < c^- < c^+ < 1$ (depending only on
$\alpha_1$ and $\alpha_2$) such that for any rectangle $R$ with side lengths $n$ and $m \in \llb \alpha_1 n,\alpha_2 n\rrb$ (i.e.\ with aspect ratio bounded away from $0$ and $\infty$) one has
\[ c^- \leq \nu_R^\xi(\mathcal{C}_v(R)) \leq c^+\]
 for any boundary conditions $\xi$, where $\mathcal{C}_v(R)$ is the event that there is a vertical open path connecting the top and bottom boundaries of $R$.
\end{theorem}
Indeed, applying the above theorem on the box $R=\llb 0,r+1 \rrb \times \llb 0,\alpha r \rrb$ with wired boundary conditions implies that $\nu_R^1(\mathcal{C}_v(R)) \leq c^+ $ for some $c^+<1$ that depends only on the constant $\alpha>0$ specified in Theorem~\ref{thm-spatial-mix}.
Introducing wired boundary at level $\alpha r$ and applying the Domain Markov property we can iterate this argument to obtain that for $R = \llb 0 ,r+1\rrb \times \llb 0,\rho r\rrb$ we have $\nu_R^1(\mathcal{C}_v(R)) \leq (c^+)^{\lfloor \rho/\alpha \rfloor} $ with the same $c^+<1$. Altogether:
\begin{corollary}\label{cor-exp-cross}
Let $r'/r > \alpha > 0$ for some $\alpha > 0$ fixed and let $\alpha \leq \rho < r'/r$. There exists some constant $\delta=\delta(\alpha)$ such that
$\nu_R^\xi(\mathcal{C}_v(R)) \leq \exp(-\delta \rho)$ for $R = \llb 0,r+1\rrb \times \llb 1,\rho r\rrb$ and any
boundary condition $\xi$, where
$\mathcal{C}_v(R)$ is the event that there is a vertical open path connecting the top and bottom of $R$.
\end{corollary}
 Since the monotonicity of boundary condition and the Domain Markov property imply that $\nu_R^1$ stochastically dominates $\nu^1_{\Lbar}(\omega(R)\in\cdot)$, we get
\begin{equation}
  \label{eq-Xi-bottom}
  \P\left(\Xi \subset \llb 0,r+1 \rrb \times \llb 0,\rho r -1\rrb \right) \geq 1 - \exp(-\delta\rho)\,.
\end{equation}

\begin{figure}
\centering \includegraphics[width=3in]{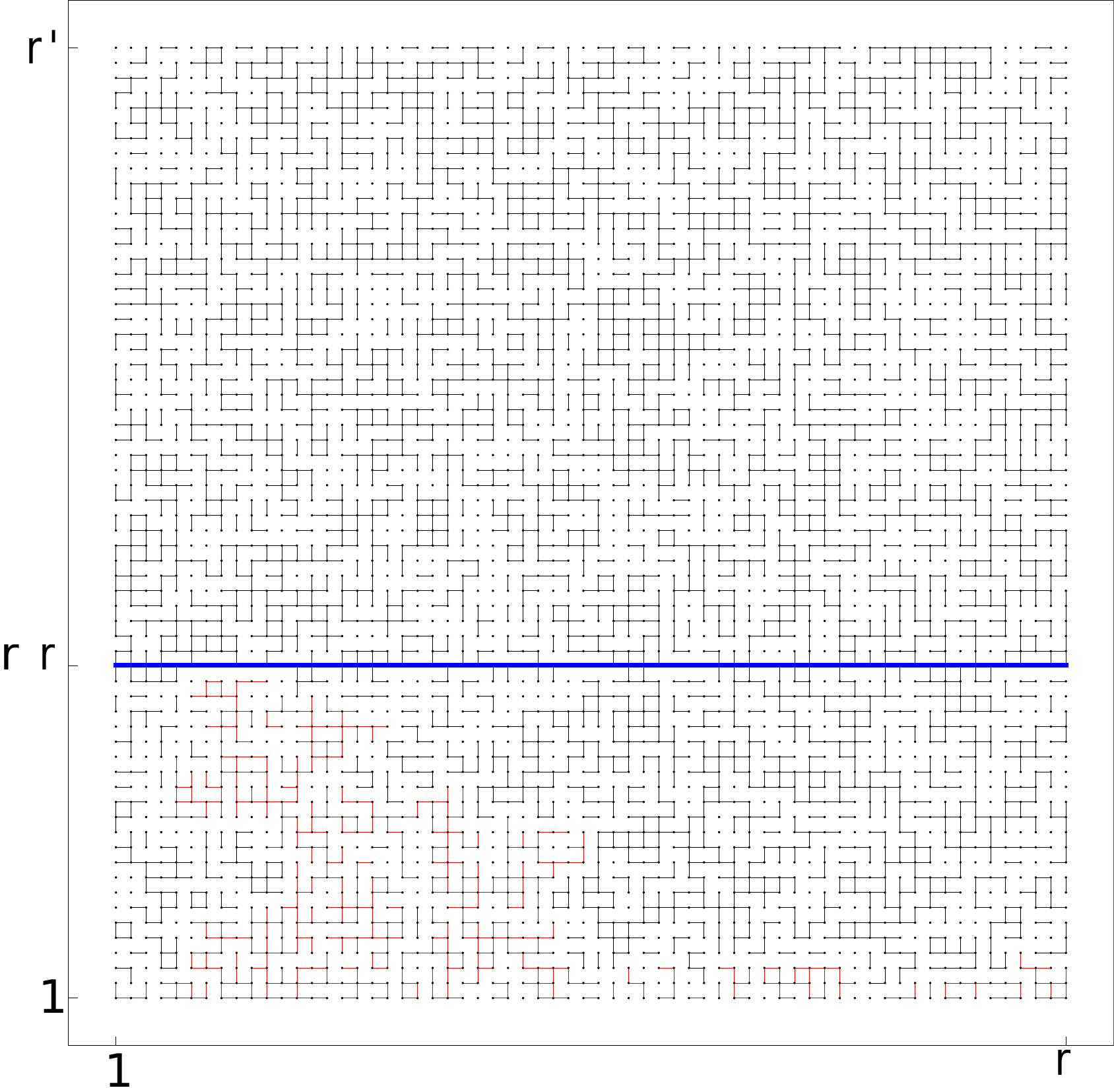}
\hspace{0.25cm} \raisebox{0.4in}{\includegraphics[width=1.5in]{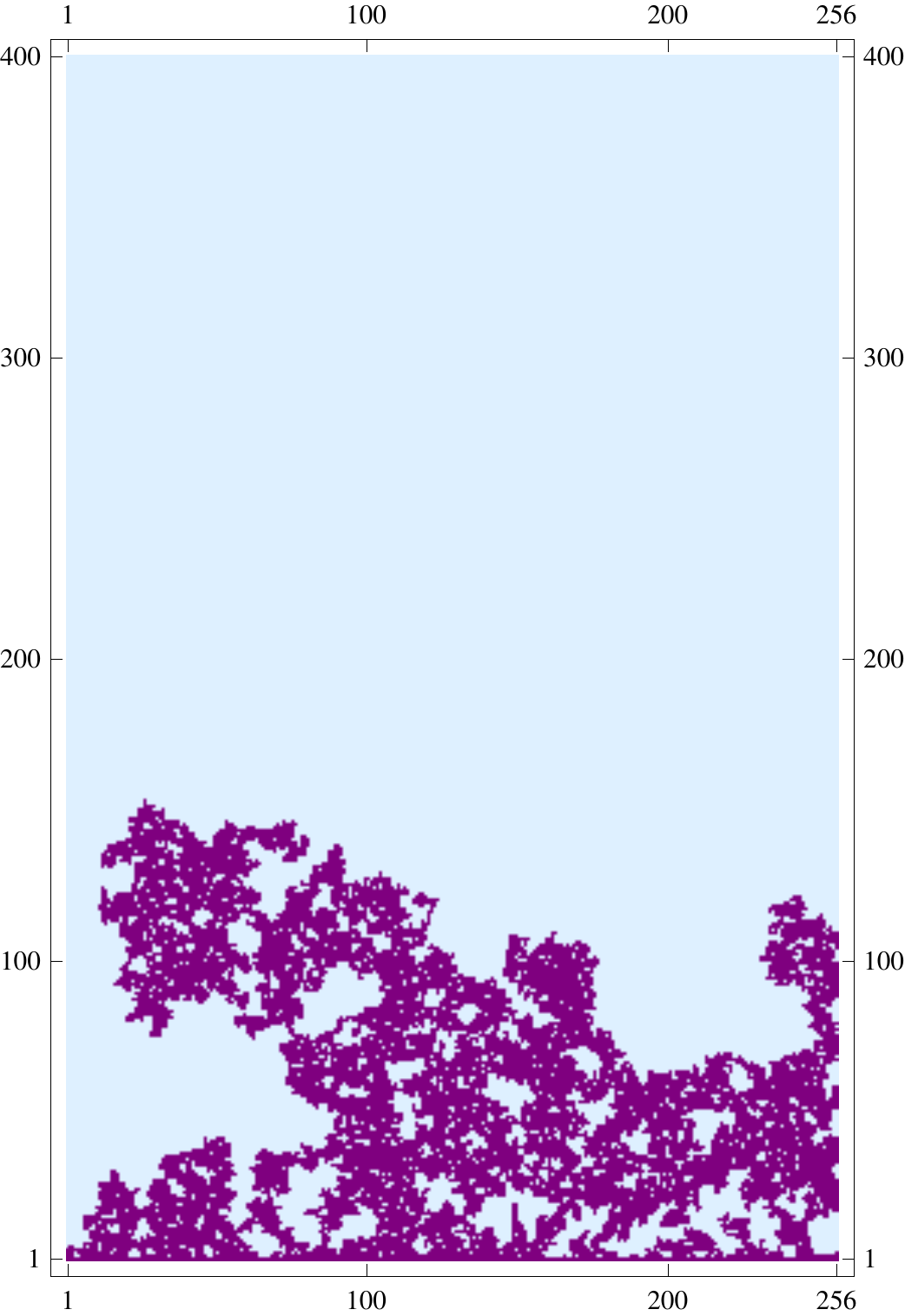}}
\caption{Critical FK-Ising representations illustrating the set $\Xi$.}
\label{fig:fkclust}
\end{figure}

Denote by $\Upsilon$ the edges that are not incident to any of the sites in $\Xi$.
We argue that conditioned on the edges revealed up to time $T$, namely $\omega^\xi(\Delta_T)$ and $\omega^\eta(\Delta_T)$, the measures $\nu_{\Lbar}^\xi(\cdot|A^\xi)$ and $\nu_{\Lbar}^\eta(\cdot |A^\eta)$ on the remaining unexposed edges $\Upsilon$ are identical. In other words,
\begin{align} &~\nu_{\Lbar}^\xi \big(\omega(\Upsilon)\in\cdot \given A^\xi\,,\, \omega(\Delta_T)=\omega^\xi(\Delta_T)\big) \nonumber\\ = &~\nu_{\Lbar}^\eta \big(\omega(\Upsilon)\in\cdot \given A^\eta\,,\, \omega(\Delta_T)=\omega^\eta(\Delta_T)\big)\,.\label{eq-xi-eta-equiv}\end{align}
To see this, first consider $\nu_{\Lbar}^\xi\left( \omega(\Upsilon)\in\cdot \given \omega(\Delta_T)=\omega^\xi(\Delta_T)\right)$. Notice that this measure over the unexposed edges $\Upsilon$ is the FK-model on the graph induced on the sites $\Xi^c$, where the edges in the interface with $\Xi$ are all closed and the boundary conditions inherited from $\xi$ are precisely the pair-wise wirings of the sites of $\cP(\xi)\cap \Xi^c$ and of those in $\cM(\xi) \cap \Xi^c$.
Crucially, by our construction (that conditioned on $A^\xi$) there is no open path between $\cP(\xi)$ and $\cM(\xi)$ in $\omega^\xi(\Delta_T)$, hence the sites $\cP(\xi)\cap \Xi^c$ are not wired to the same cluster of $\cM(\xi)\cap \Xi^c$ at time $T$.

An analogous statement holds for $\omega^\eta(\Delta_T)$, and since $\xi,\eta$ differ only on $\Gamma$ we infer that
$\nu_{\Lbar}^\tau\left( \omega(\Upsilon)\in\cdot \given \omega(\Delta_T)=\omega^\tau(\Delta_T)\right)$ has the same boundary conditions for both $\tau\in\{\xi,\eta\}$. The Domain Markov property now gives
\[ \nu_{\Lbar}^\xi\big( \omega(\Upsilon)\in\cdot \given \omega(\Delta_T)=\omega^\xi(\Delta_T)\big) =
\nu_{\Lbar}^\eta\big( \omega(\Upsilon)\in\cdot \given \omega(\Delta_T)=\omega^\eta(\Delta_T)\big)\,.\]
Furthermore, since there was no open path $\cP \rightsquigarrow \cM$ neither in $\omega^\xi(\Delta_T)$ nor in $\omega^\eta(\Delta_T)$ and the interface between $\Xi,\Xi^c$ is closed, the final configuration $\omega^\tau$ for $\tau\in\{\xi,\eta\}$ can only have such a path if it is contained in $\Upsilon$, which together with the last equality implies \eqref{eq-xi-eta-equiv}.

Observe that, as a result of \eqref{eq-xi-eta-equiv}, beyond time $T$ our procedure automatically couples the configurations $\omega^\xi,\omega^\eta$ via the identity coupling. Altogether, we have shown that we can couple $\omega^\xi,\omega^\eta$ on $\Upsilon$, in particular obtaining that every component $C \subset \Xi^c$ is identical between these two configurations, and any component $C'$ intersecting $C$ is completely contained in $C$ (since the edges between $\Xi$ and $\Xi^c$ are all closed).

We may now assign spins according to Corollary~\ref{cor-FK-Ising-bc}:
Since there is no path between $\cP,\cM$ in either of the configurations, we can assign plus-spins to $\cP$ and minus-spins to $\cM$ when generating the Ising configurations $\sigma^\xi,\sigma^\eta$. Coupling the i.i.d.\ spin values for the remaining clusters of $\Xi^c$ to be the same in both configurations gives $\sigma^\xi(\Xi^c) = \sigma^\eta(\Xi^c)$. Corollary~\ref{cor-FK-Ising-bc} further implies that $\sigma^\tau \sim \mu_\Lambda^\tau$ for both $\tau\in\{\xi,\eta\}$.

 The proof is concluded by inequality~\eqref{eq-Xi-bottom} which ensures that $\Lambda_\btop \subset \Xi^c$ with probability at least $1 - \exp(-\delta\rho)$.
\qed

\begin{remark*}
The exponent $C>0$ in Theorem~\ref{thm-1} can be readily made explicit in terms of the crossing probabilities in the fully-wired critical FK-model.
For instance, in the setting of the square lattice of side-length $n$ the proof gives \[C = 2\log_{3/2}\left[2/(1-p^+(\tfrac13))\right]\,,\]
where $p^+(\tfrac13)$ is the limiting vertical crossing probability in the FK-model on a fully-wired rectangle with conformal modulus $\frac13$.
Given the limiting function $p^+$ for all conformal moduli one could improve the resulting exponent
by optimizing the overlap between the blocks in the recursive analysis. See \cites{LLS,LPS} for numerical estimates of the function $p^+$.
\end{remark*}

\section{A polynomial lower bound and other boundary conditions}
In this section we establish a lower bound on the inverse-gap of the Glauber dynamics for the critical Ising model on a square lattice and extend the upper bound of Theorem~\ref{thm-1} to periodic/free boundary conditions.

\subsection{A polynomial lower bound on the inverse-gap}
To complement our result that the inverse-gap is bounded from above by a polynomial in the side-length $n$ we show the following:
\begin{theorem}
  \label{thm-gap-lower-bound}
  Let $\gap^\xi_\Lambda$ be the spectral-gap of the Glauber dynamics for the critical Ising model on a square lattice $\Lambda$ of side-length $n$ with an arbitrary boundary condition $\xi$.
  Then $\big(\gap_\Lambda^\xi\big)^{-1} \geq c n^{7/4}$ for some absolute $c>0$.
  Furthermore, this also holds for rectangles with shorter side-length $n$.
\end{theorem}
\begin{proof}
In the static Ising model, as the temperature decreases from high to  critical, Onsager~\cite{Onsager} showed that the correlation
of two spins transitions from having an exponential decay to a polynomial one. As mentioned in the introduction, the dynamics also exhibits such a critical slowdown.
At criticality, Holley~\cite{Holley1} showed that in the infinite volume lattice the spin-spin autocorrelation $\E \sigma_0(0)\sigma_t(0)$
decays like $t^{-1/4}$ rather than exponentially. This can be translated to a polynomial lower bound on $\gap^{-1}$ that is slightly worse than linear. Here we will use the spin-spin correlations to obtain the $n^{7/4}$ lower bound.

Consider the Ising model on $\Lambda=\llb 1, n\rrb^2$ with some boundary condition $\xi$ and let $\Lambda^* = \llb n/4, 3n/4 \rrb^2$.
Our test-function $f$ will be the magnetization over $\Lambda^*$, i.e.\ $f(\sigma) = \sum_{x\in\Lambda^*} \sigma(x)$.
Recalling the definition of the Dirichlet form $\mathcal{E}(f)$ in~\eqref{eq-def-E(f)}, we clearly have $\mathcal{E}(f) = O(n^2)$ since $f$ is $1$-Lipschitz. It thus remains to estimate $\var_{\mu^\xi_\Lambda}(f)$.

Analogous to our approach in the previous section, we construct an sample of the Ising model at equilibrium by going through the critical FK-model via the Edwards-Sokal coupling. We first reveal the set of open bonds in the FK-model on $\Lambda \setminus \Lambda^*$ (that is, every bond with at least one endpoint outside of $\Lambda^*$) and let $\cF$ denote the $\sigma$-algebra generated by these variables.
This is followed by revealing the remaining FK-configuration and constructing the Ising configuration via the Edwards-Sokal coupling. When revealing the FK-configuration we condition on the event $A^\xi$, i.e.\ that there is no open path connecting plus and minus sites of $\xi$.

By the total-variance formula we have that
\begin{align*}
\var(f) &= \E \var(f\mid \cF) + \var \E[f\mid \cF] \geq \E \var(f\mid \cF) \one_{ \{\Lambda^*\not\rightsquigarrow \partial\Lambda\}} \,,
\end{align*} where $\{\Lambda^*\not\rightsquigarrow \partial\Lambda\}$ denotes the event that there is no open path connecting $\Lambda^*$ to $\partial\Lambda$ (i.e., there is a circuit of open edges surrounding $\Lambda^*$ in the dual FK-configuration). As observed by~\cite{DHN}, this event has probability at least $c_1>0$ irrespective of the boundary condition $\xi$ without conditioning on $A^\xi$. By the FKG inequality, this conditioning only increases the probability of the event $\{\Lambda^*\not\rightsquigarrow \partial\Lambda\}$.

Observe that conditioned on the above mentioned event, any site $x\in\Lambda^*$ is disconnected from $\partial\Lambda$ and so in the final Ising configuration its expected magnetization is $\E \sigma(x) = 0$.

Onsager's decay of correlation result was reobtained by~\cite{DHN} using their RSW-type inequalities for the FK-model. In particular, they established that if $\Lambda^*$ is a box of side-length $m$ and $x,y\in\Lambda$ have distance at least $\epsilon m$ from $\partial\Lambda$ for some fixed $\epsilon>0$, then for any boundary condition $\eta$
\begin{equation}\label{e:corrDecay}
\nu^\eta_{\Lambda^*}(x \leftrightsquigarrow y) \geq c_2 m^{-1/4}
\end{equation}
for some $c_2=c_2(\epsilon)> 0$. In our setting this implies that
\begin{align*}
\var(f  \mid \cF, \{\Lambda^*\not\rightsquigarrow \partial\Lambda\}) &= \sum_{x,y\in\Lambda^*} \E [\sigma(x)\sigma(y)\mid \cF, \{\Lambda^*\not\rightsquigarrow \partial\Lambda\}] \\
&\geq \sum_{x,y\in\Lambda^*} \P (x\rightsquigarrow y \mid \cF, \{\Lambda^*\not\rightsquigarrow \partial\Lambda\}) \geq c_3 n^{15/4}\,.
 \end{align*}
 Combining this with the lower bound $c_1>0$ on the probability of $\{\Lambda^*\not\rightsquigarrow \partial\Lambda\}$ we deduce that $\var(f) \geq c_4 n^{15/4}$ for some absolute $c_4 > 0$.
Plugging $f$ as a test-function for the spectral-gap in~\eqref{eq-dirichlet-form} we now deduce that $\gap^\xi_\Lambda \leq c n^{-7/4}$ for some absolute $c>0$, as required.
\end{proof}

\subsection{Free or periodic boundary conditions}
The next theorem establishes polynomial mixing for free/periodic boundary conditions. In fact, the same method infers this result for any mixed boundary conditions, e.g.\ periodic on one side and a mixture of free and fixed conditions on the other. In what follows we restrict our attention to boxes of bounded aspect-ratio although our arguments from the previous section can be applied to achieve analogues of Theorem~\ref{thm-inverse-gap}.
\begin{theorem}
  \label{thm-free-periodic}
  Let $\gap^\xi_\Lambda$ be the spectral-gap of the Glauber dynamics for the critical Ising model on a square lattice $\Lambda$ of side-length $n$ under free or periodic boundary conditions.
  Then $\big(\gap_\Lambda^\xi\big)^{-1} \leq n^C$ for some absolute constant $C>0$.
\end{theorem}
\begin{proof}
The proof of Theorem~\ref{thm-1} holds with slight modifications in the current setting and in what follows we describe the required adjustments.

 Consider first the case of free boundary conditions. Through the course of the recursive analysis, some of the boundaries of the blocks $\Lambda_1,\Lambda_2$ may be free while others have a fixed (arbitrary) boundary condition (e.g., in the first step there are 3 sides with free boundary). Hence, the entire proof holds given the following variant of the spatial mixing statement of Theorem~\ref{thm-spatial-mix}:

 \begin{theorem}\label{thm-spatial-mix-free}
 The statement of Theorem~\ref{thm-spatial-mix} holds when the given boundary conditions $\xi,\eta$ are possibly free on one or more of the sides of $\partial\Lambda$.
\end{theorem}
The proof of the above theorem proceeds almost exactly the same as that of Theorem~\ref{thm-spatial-mix}, where the only essential difference is that a side with free boundary now belongs to $\Lambda$ rather than to $\Lbar$.

Notice that the exponent $C>0$ obtained in the above proof for free boundary is identical to the one from the proof of Theorem~\ref{thm-1}.

It remains to treat periodic boundary conditions. In the Edwards-Sokal coupling of the Ising and FK-Ising models, a periodic boundary condition corresponds precisely to wiring the two identified boundary sites. Thus, we may immediately infer another analogue of the spatial-mixing result specialized to two possibly disagreeing sides under the boundary conditions, with the agreeing sides possibly having a periodic boundary condition:
\begin{theorem}\label{thm-spatial-mix-periodic}
Let $\Lambda = \llb 1,r \rrb \times \llb 1 , r'\rrb$ for some integers $r,r'$ satisfying $\alpha_1 \leq r'/r \leq \alpha_2 $ with $\alpha_1,\alpha_2 > 0$ fixed and let
$\Lambda_\bmid= \llb 1,r \rrb \times \llb \psi r',(1-\psi)r' \rrb$ for some fixed $0<\psi<\frac12$.
Let $\xi,\eta$ be two boundary conditions on $\Lambda$ that differ only on the top and bottom boundaries $\llb 1,r\rrb \times (\{0\} \cup \{r'+1\})$ and possibly have periodic boundary on the remaining sides.
Then \[ \left\| \mu^\xi_\Lambda(\sigma(\Lambda_\bmid)\in\cdot) - \mu^\eta_\Lambda(\sigma(\Lambda_\bmid)\in\cdot) \right\|_\tv \leq 1-\delta\,,\]
where $\delta > 0$ is a constant that depends only on $\alpha_1,\alpha_2,\psi$.
\end{theorem}
To prove the above result, apply the following variant of the argument of the proof of Theorem~\ref{thm-spatial-mix} on the corresponding FK-Ising model.
Instead of exposing the set $\Xi$ comprising the sites connected to the bottom boundary $\Gamma$ by an open path, we simultaneously expose $\Xi,\Xi'$ where $\Xi'$ is its analogue with respect to the top boundary. The same reasoning now implies that $\Xi \cup \Xi' \subset \Lambda^c_\bmid$ with positive probability, allowing the coupling of the two FK-measures corresponding to the boundaries $\xi,\eta$.

With the above estimate at hand, consider first boxes with one pair of periodic boundary conditions. Using the block-dynamics with recursive analysis as in Theorem~\ref{thm-1}, this time we choose each of the two blocks to contain both of the periodic boundaries, e.g.\
when these are the top and bottom take
\[\Lambda_1 = \llb 1,r \rrb \times \left(\llb \tfrac13 r' , r'\rrb\cup\llb 1,\tfrac19 r'\rrb\right) ~,~ \Lambda_2= \llb 1,r \rrb \times \left(\llb 1 , \tfrac23 r'\rrb\cup\llb \tfrac89r',r'\rrb\right)\,.\]
After one application of this argument (notice that we did not yet make use of the possible periodic boundary conditions in the spatial-mixing result) we arrive at rectangles with bounded aspect ratio and fixed boundary conditions, where Theorem~\ref{thm-1} already applies. We deduce that the inverse-gap for rectangles with bounded aspect-ratio and one pair of periodic boundaries is polynomial in the side-length.

Finally, to obtain this result for tori (corresponding to periodic boundary conditions on all sides), split the torus into two overlapping blocks, each with one pair of periodic boundary conditions (in the same manner detailed above). At this point, the spatial-mixing result (here applied to periodic boundary) asserts that with positive probability the complement of one block is coupled within a single step. This in turn implies that the inverse-gap of the block-dynamics is uniformly bounded, and the aforementioned bound on the single-site dynamics (for boxes with a single pair of periodic boundaries) completes the proof.
\end{proof}

\subsection{Critical anti-ferromagnetic Ising model}
It is well-known that the anti-ferromagnetic Ising model on the square lattice is equivalent to the ferromagnetic model on the lattice with modified boundary conditions (via the transformation that flips the spins at all odd sites, including those in the boundary). As our bounds hold for any boundary condition we arrive at the following:
\begin{corollary}\label{cor-antiferro}
  The inverse-gap of the Glauber dynamics for the critical anti-ferromagnetic Ising model on the square lattice of side-length $n$ under arbitrary boundary condition is polynomial in $n$.
\end{corollary}

\section{Concluding remarks and open problems}

\begin{list}{\labelitemi}{\leftmargin=2em}
\item In this work we have established that the inverse-gap of the dynamics for the Ising model on the square lattice is polynomial in its side-length, with a bound independent of the boundary condition. The proof hinges on the recent breakthroughs in the understanding of the critical FK-representation for the Ising model.
\item Furthermore, we show that on rectangles with different side-lengths (whose ratio is not necessarily bounded), the inverse-gap is bounded by a polynomial of its shorter side-length only.
\item A natural question that arises from this work is to determine the exponent of the polynomial growth of the inverse-gap for the critical Ising model in $\Z^2$. At the present time this remains a formidable challenge.
\item Another enticing open problem would be to obtain an upper bound on the inverse-gap for the critical Ising model in higher dimensions ($\Z^d$ for $d\geq3$). Here the machinery of SLE is no longer available necessitating new ideas for the sought after spatial mixing properties of the model.
\end{list}

\end{document}